\documentclass[10pt]{amsart}

\newtheorem{corollary}{Corollary}[section]

\newtheorem{lemma}[corollary]{Lemma}
\newtheorem{proposition}[corollary]{Proposition}
\newtheorem{remark}[corollary]{Remark}
\newtheorem{rmks}[corollary]{Remarks}
\newtheorem{theorem}[corollary]{Theorem}
\newcommand{\mylabel}[1]{\label{#1}
            \ifx\undefined\stillediting
            \else \fbox{$#1$}\fi }
\newcommand{\BE}{\begin{equation}}

\newcommand{\EEQ}{\end{equation}}
\newcommand{\rfb}[1]{\mbox{\rm
   (\ref{#1})}\ifx\undefined\stillediting\else:\fbox{$#1$}\fi}

\newfont{\Blackboard}{msbm10 scaled 1200}

\newfont{\roma}{cmr10 scaled 1200}

\def\CC{\rm \hbox{C\kern-.56em\raise.4ex
         \hbox{$\scriptscriptstyle |$}\kern+0.5 em }}

\newcommand{\mm}    {{\hbox{\hskip 0.5pt}}}

\newcommand{\nm}    {{\hbox{\hskip -3pt}}}

\newcommand{\bluff} {{\hbox{\raise 15pt \hbox{\mm}}}}

\newcommand{\FORALL} {{\hbox{$\hskip 11mm \forall \;$}}}

\makeatletter
\def\section{\@startsection {section}{1}{\z@}{-3.5ex plus -1ex minus
    -.2ex}{2.3ex plus .2ex}{\large\bf}}
\makeatother
\def\be{\begin{equation}}
\def\ee{\end{equation}}

\def\ds{\displaystyle}
\begin{document}
\renewcommand{\baselinestretch}{0.9}

\date{\today}
{


\title[Non uniform decay of the energy]{Non uniform decay of the energy of some dissipative
evolution systems}

\author{Ka\"{\i}s Ammari}
\address{UR Analysis and Control of PDEs, UR13ES64, Department of Mathematics, Faculty of Sciences of Monastir,
University of Monastir, 5019 Monastir, Tunisia}
\email{kais.ammari@fsm.rnu.tn}

\author{Ahmed Bchatnia}
\address{UR  Analyse Non-Lin\'eaire et G\'eom\'etrie, UR13ES32, Department of Mathematics, Faculty of Sciences of Tunis, University of Tunis El Manar, Tunisia}
\email{ahmed.bchatnia@fst.rnu.tn} 

\author{Karim El Mufti}
\address{UR Analysis and Control of PDEs, UR13ES64, ISCAE, University of Manouba, Tunisia}
\email{karim.elmufti@iscae.rnu.tn}

\begin{abstract}
In this paper we consider second order evolution equations with
bounded damping. We give a characterization of a non uniform decay
for the damped problem using a kind of observability estimate for
the associated undamped problem.
\end{abstract}
\date{}

\maketitle

\tableofcontents

\section{Introduction and main results}
Let $X$ be a complex Hilbert space with norm and inner product
denoted respectively by $||.||_{X}$ and $\left\langle .,.\right\rangle_{X}$.
Let $A$ be a linear unbounded self-adjoint and  strictly
positive operator in
$X$ and $V ={\mathcal{D}}(A^{\frac{1}{2}})$ be the domain of
$A^{\frac{1}{2}}$, with $$\|x\|_{V} = \|A^{\frac{1}{2}}x\|_{X}, \forall x \in V.$$ \\
Denote by $({\mathcal{D} }(A^{\frac{1}{2}}))^{\prime}$
the dual space of ${\mathcal{D}}(A^{\frac{1}{2}})$ obtained by means of
the inner product in $X$.
\noindent Further, let $U$ be a complex Hilbert space (identified to
its dual) and $B \in {\mathcal{L}}(U,X)$. \\
\noindent
Most of the linear control problems coming from elasticity can be written as
\be
\label{nom}
\left\{
\begin{array}{ll}
w^{\prime \prime }(t) + Aw(t) + Bu(t) = 0, \\
\nm
\ds
w(0) = w_0, \, w^{\prime}(0) = w_1,
\end{array}
\right. \ee where $w : [0,T] \rightarrow X$ is the state of the
system, $u \in L^2(0,T;U)$ is the input function and denote the
differentiation with respect
to time by ``$\prime$''. \\

\noindent
We define the energy of $w(t)$ at instant $t$
by
\[
 E(w(t))=
\frac{1}{2} \Bigr\{ ||w^{\prime}(t)||_{X}^2 + ||A^{\frac{1}{2}}w(t)||_X^2 \Bigr\}.
\]
Simple formal calculations give \be \label{ESTENb} E(w(0))-E(w(t))=
\int_0^t \left\langle Bu(s),w^{\prime}(s)\right\rangle_{X} \,
ds,\FORALL t\ge 0. \ee This is why, in many problems coming in
particular from elasticity, the input $u$ is given in the feedback
form $u(t) = B^* w^{\prime}(t)$, which obviously gives a
nonincreasing energy and which corresponds to collocated actuators
and sensors.

The aim of this paper is to give sufficient and necessary conditions
on the conservative system (\ref{eq3}) making the corresponding
closed loop system \be \label{eq1} \left\{
\begin{array}{ll}
w^{\prime \prime}(t) + Aw(t) + BB^*w^\prime(t) = 0, \\
w(0) = w_0, w^\prime(0) = w_1,
\end{array}
\right.\ee non uniformly stable. The strategy to get such a decay
rate will consist to generalize a kind of observability
estimate given in \cite{phung1}.\\
Any sufficiently smooth solution of \rfb{eq1} satisfies the energy estimate
\be
E(w(0))-E(w(t))=\int_0^t ||B^*w^{\prime}(s)||_{U}^2ds,\FORALL t\ge 0.
\label{ESTEN}\ee
In particular \rfb{ESTEN} implies that
$$ E(w(t))\le E(w(0)),\ \forall t\ge 0.$$
In the natural well-posedness space $V\times X$, the existence and uniqueness of finite energy solutions of
\rfb{eq1} can be obtained by standard semi-group methods.\\

Denote by $\phi$ the solution of the associated undamped problem
\be\label{eq3} \left\{
\begin{array}{ll}
\phi^{\prime \prime }(t) + A \phi(t) = 0, \\
\phi(0) = w_0, \, \phi^{\prime}(0) = w_1.
\end{array}
\right.\ee
It is well known that \rfb{eq3} is well-posed in ${\mathcal{D}}(A) \times V$
and in $V \times X$.\\
Our main result is stated as follows:\\
Let $\mathcal{G}$ be a continuous positive increasing real function
on $[0,+\infty)$ and define the function $\mathcal{F}$ by
$\mathcal{F}(x)= x \, (\mathcal{G}(x))^2$.
\begin{theorem} \label{prop1}
\begin{enumerate}
\item
Assume that there exists $C>0$ such that for all non-identically zero initial
data $(w_0,w_1)\in V \times X $ and for all $t>0$, the solution $w$ of \rfb{eq1}
satisfies:
\begin{equation}\label{depol}
E(w(t))\leq C||(w_{0},w_{1})||_{ V\times X}^2\mathcal{G}^{-1}\left(\frac{1}{t}\right),
\end{equation}
then there exists $C>0$ such that the solution $\phi$ of
\textsc{(\ref{eq3})} satisfies:
\begin{equation}\label{obs1}
||(w_{0},w_{1})||_{V\times X}^2\leq
16\int_{0}^{\frac{1}{\mathcal{G}\left(\frac{1}{2C \Lambda}\right)}}\|B^* \phi^{\prime}(t)\|^2_{U} \, dt,
\end{equation}
where $$\Lambda=\frac{||(w_{0},w_{1})||_{ \mathcal{D}(A)\times V}^2}{||(w_{0},w_{1})||_{V\times X}^2}.$$
\item Assume that $x\mapsto x \, \mathcal{F}^{-1}(\frac{1}{x})$ is an increasing function
and there exists $C>0$ such that for all non-identically zero initial data
$(w_{0},w_{1})\in  \mathcal{D}(A)\times V $, the solution $\phi$ of
\textsc{(\ref{eq3})} satisfies:
\begin{equation}\label{obs2}
||(w_{0},w_{1})||_{V\times X}^2\leq
C\int_{0}^{\frac{1}{\mathcal{G}\left(\frac{1}{2C\Lambda}\right)}}\|B^* \phi^\prime(t)\|^2_{U} dt.
\end{equation}
Then there exists $C>0$ such that for all $t>0$, the solution $w$ of \rfb{eq1}
satisfies:
\begin{equation}\label{decay}
E(w(t))\leq C||(w_{0},w_{1})||_{ \mathcal{D}(A)\times V}^2\mathcal{F}^{-1}\left(\frac{1}{\sqrt{t}}\right).
\end{equation}
\end{enumerate}
\end{theorem}

\begin{corollary}
The weak observability \footnote{see \cite{ammari,kaisserge,nodea} for more details} i.e. there exist $T, C>0$
such that for all
$(w_0,w_1) \in V \times X $  the solution $\phi$ of (\ref{eq3}) satisfies
\be
\int_{0}^{T} ||B^* \phi^{\prime}(t)||^2_{U} \, dt \geq C \,
||(w_0,w_1)||^2_{V \times X} \, {\mathcal G} \left(\frac{||(w_0,w_1)||^2_{X \times ({\mathcal{D} }(A^{\frac{1}{2}}))^{\prime}}}
{||(w_0,w_1)||^2_{V \times X}} \right),
\label{nunif1}
\ee
implies in particular (\ref{obs1}).
\end{corollary}
The paper is organized as follows: In section \ref{sec2} we prove
our main result and in the last section we give some applications
both in the linear and the nonlinear case. Decay rates for nonlinear
dissipations were obtained under our generalized observability
estimate. Here we mention that the literature is less provide. We
cite essentially \cite{am-al,phung2}.
\section{Proof of Theorem \ref{prop1}}\label{sec2}
The following lemma will be very useful.
\begin{lemma}\label{lemme2}
Let $\mathcal{H}$ (resp. $\mathcal{G}$) be a continuous positive decreasing (resp. increasing) real function on $[0,+\infty)$. Suppose that $\mathcal{H}$ is bounded by one and there exists a positive constant $c$ such that
\begin{equation}\label{h}
\mathcal{H}(s)\leq \frac{c}{\big(\mathcal{G(H}(s))\big)^2}\left(\mathcal{H}(s)-\mathcal{H}\left(\frac{1}{\mathcal{G(H}(s))}+s\right)\right),\ \forall s>0.
\end{equation}
Suppose that $x\mapsto x \, \mathcal{F}^{-1}(\frac{1}{x})$ is an increasing function, then there exists $C > 0$ such that for any $t > 0$,
\begin{equation}
\mathcal{H}(t)\leq C\mathcal{F}^{-1}\left(\frac{1}{\sqrt{t}}\right).
\end{equation}
\end{lemma}
The proof is similar to that of Lemma B in \cite{phung1}. Since our result is more general, we give it for the reader's convenience.
\begin{proof}
Let $t>0$. We distinguish two cases:\\

\noindent
$\bullet$ If
\begin{equation}\label{hb}
\mathcal{G(H}(s))<\frac{1}{t}, \mbox{ then }  \qquad \mathcal{H}(s)\leq \mathcal{G}^{-1}\left(\frac{1}{t}\right).
\end{equation}\label{h(t+s)}
$\bullet$ If
\begin{displaymath}
\mathcal{G(H}(s))\geq\frac{1}{t}, \qquad then \qquad \frac{1}{\mathcal{G(H}(s))}+s\leq t+s,
\end{displaymath}
therefore
\begin{equation}
\mathcal{H}(t+s)\leq \mathcal{H}\left(\frac{1}{\mathcal{G(H}(s))}+s \right)
\end{equation}
and we get
\begin{equation}\label{h(t+s)b}
\mathcal{F}(\mathcal{H}(s))\leq \mathcal{H}(s)-\mathcal{H}(t+s).
\end{equation}
The inequalities (\ref{hb}) and (\ref{h(t+s)b}) give
\begin{equation}
\mathcal{H}(s)\leq \mathcal{F}^{-1}(\mathcal{H}(s)-\mathcal{H}(t+s)) +\mathcal{G}^{-1}\left(\frac{1}{t}\right), \forall s, t>0.
\end{equation}

We introduce the function $\Psi_{t}$ defined on $]0,+\infty[$ by:
\begin{equation}
\Psi_{t}(s)=\frac{1}{\mathcal{F}^{-1}\left(\frac{t}{s}\right) +\mathcal{G}^{-1}\left(\frac{1}{t}\right)}.
\end{equation}
We distinguish two cases:\\

\noindent
$\bullet$ If $ \mathcal{H}(s)-\mathcal{H}(t+s)< \frac{t}{t+s} $  then $ \mathcal{H}(s)\leq \Psi_{t}(t+s)$ and we deduce
\begin{equation}
\Psi_{t}(t+s)\mathcal{H}(t+s)\leq1.
\end{equation}
$\bullet$ If $ \mathcal{H}(s)-\mathcal{H}(t+s)> \frac{t}{t+s} .$ Taking into account that $\mathcal{H}(s)\leq 1$  we obtain
$$\frac{t}{t+s}\mathcal{H}(s)\leq\frac{t}{t+s},$$
so
\begin{equation}
\frac{t}{t+s}\mathcal{H}(s)\leq\frac{t}{t+s}<\mathcal{H}(s)-\mathcal{H}(t+s),
\end{equation}
and we deduce
\begin{equation}
\mathcal{H}(t+s)<\mathcal{H}(s)\, \frac{s}{t+s}.
\end{equation}
Consequently,
\begin{eqnarray}
\Psi_{t}(t+s)\mathcal{H}(t+s)\nonumber
&<&\Psi_{t}(t+s)\mathcal{H}(s)\, \frac{s}{t+s} \\\nonumber
&=&\frac{\Psi_{t}(t+s)}{t+s}\mathcal{H}(s)\Psi_{t}(s)\frac{s}{\Psi_{t}(s)}\\
&=&\mathcal{H}(s)\Psi_{t}(s)\, \frac{\frac{\Psi_{t}(t+s)}{t+s}}{\frac{\Psi_{t}(s)}{s}}.
\end{eqnarray}
Using the increasing property of $x\mapsto x{\mathcal{F}}^{-1}(\frac{1}{x})$, we obtain
\begin{equation}
\Psi_{t}(t+s)\mathcal{H}(t+s)<\Psi_{t}(s)\mathcal{H}(s).
\end{equation}
We have proved that for all $s,t>0$, we have either
$$\Psi_{t}(t+s)\mathcal{H}(t+s)\leq1 \; \hbox{or} \; \Psi_{t}(t+s)\mathcal{H}(t+s)<\Psi_{t}(s)\mathcal{H}(s).$$
In particular, we deduce that for any $t>0$ and $n\in \mathbb{N}^*,$ either
\begin{equation}
\Psi_{t}((n+1)t)\mathcal{H}((n+1)t)\leq1 \mbox{ or } \Psi_{t}((n+1)t)\mathcal{H}((n+1)t)<\Psi_{nt}(t)\mathcal{H}(nt).
\end{equation}
Hence, we have
\begin{equation}
\Psi_{t}((n+1)t)\mathcal{H}((n+1)t)\leq max(1,\Psi_{t}(t)\mathcal{H}(t))=1.
\end{equation}
Therefore, for all $t>0$ and $n\in\mathbb{N}^{*}$,
\begin{equation}
\mathcal{H}((n+1)t)\leq \mathcal{F}^{-1}\left(\frac{1}{n+1}\right) +\mathcal{G}^{-1}\left(\frac{1}{t}\right).
\end{equation}
Choose $n$ such that $n + 1\leq t<n + 2$ and make use again of the increasing property of $x\mapsto x{\mathcal{F}}^{-1}(\frac{1}{x})$, we get for all $t\geq 2$ :

\begin{equation}
\mathcal{H}(t^2)\leq \mathcal{F}^{-1} \left( \frac{1}{t} \right) +\mathcal{G}^{-1} \left( \frac{1}{t} \right).
\end{equation}
Since $|\mathcal{G}^{-1}(x)|\leq |\mathcal{F}^{-1}(x)|$ close to zero, the desired result follows immediately.
\end{proof}

After, we give the proof of the main result.
\begin{proof}[Proof of (1)]
We combine (\ref{depol}) and the following formula:
\begin{equation}\label{wt}
E(w(t))=E(\phi(0))-2\int_{0}^{t}\|B^* \phi^\prime(t)\|^2_{U}  ds \hspace{0.4cm } \forall t>0,
\end{equation}
to get:
\begin{equation}
E(\phi(0))-2\int_{0}^{t}\|B^* \phi^\prime(t)\|^2_{U} \, ds \leq
C\mathcal{G}^{-1}\left(\frac{1}{t}\right)||(w_{0},w_{1})||_{\mathcal{D}(A)\times V}^2.
\end{equation}
Take $t=\frac{1}{\mathcal{G}\left(\frac{1}{2C\Lambda}\right)},$ we obtain
\begin{equation}
E(\phi(0))-2\int_{0}^{\frac{1}{\mathcal{G}\left(\frac{1}{2C\Lambda}\right)}}\|B^* \phi^\prime(t)\|^2_{U} ds \leq
\frac{1}{2}||(w_{0},w_{1})||_{V\times X}^2.
\end{equation}

We deduce that
\begin{equation}\label{444}
||(w_{0},w_{1})||_{V\times X}^2\leq
4\int_{0}^{\frac{1}{\mathcal{G}(\frac{1}{2C\Lambda})}}\|(B^* \phi)^{\prime}(t)\|^2_{U}ds.
\end{equation}
Now, let us consider $v=\phi-w$, then $v$ satisfies the following system:
\begin{equation}\label{second}
\left \{
\begin{array}{ll}
v^{\prime \prime}(t) + Av(t)+ BB^* v^\prime (t) =BB^\star \phi^\prime (t), \, t > 0,\\
(v(0),v^\prime(0)) = (0,0).
             \end{array}\right.
\end{equation}
\\
Multiply the first equation of (\ref{second}) by $v^\prime$,
and integer by parts to get
\begin{equation}\label{aaa}
E(v(t))+2\int_{0}^{t}\|B^* v^\prime(s)\|^2_{U} \,
ds = 2\int_{0}^{t}\langle B^*\phi^\prime (s),B^* v^\prime(s)\rangle_U \, ds.
\end{equation}

Make use of Young inequality,
\begin{equation}
E(v(t))+2\int_{0}^{t}\|B^* v^\prime(s)\|^2_{U} \,
ds \leq\int_{0}^{t}\left(\|B^* \phi^\prime(s)\|^2_{U} + \|B^* v^\prime(t)\|^2_{U} \right) \,
ds.
\end{equation}
Hence,
\begin{equation}\label{36}
E(v(t)) + \int_{0}^{t} \|B^* v^\prime(s)\|^2_{U} \,
ds \leq \int_{0}^{t}\|B^* \phi^\prime(s)\|^2_{U}
ds.
\end{equation}
Since
$$\|B^*w^\prime\|^2_U=\|B^*\phi^\prime - B^*v^\prime\|^2_U \leq 2 \left(\|B^*\phi^\prime\|^2_U +
\|B^*v^\prime\|^2_U \right),$$ therefore
\begin{equation}\label{37}
 \int_{0}^{\frac{1}{\mathcal{G}\left(\frac{1}{2C\Lambda}\right)}}
\|B^* w^\prime(t)\|^2_{U} dt\\
\leq
2 \int_{0}^{\frac{1}{\mathcal{G}\left(\frac{1}{2C\Lambda}\right)}}\left(\|B^*\phi^\prime\|^2_U+
\|B^*v^\prime\|^2_U\right) \, dt.
\end{equation}
Thanks to (\ref{36}) and (\ref{37}), we obtain
\begin{equation}
\int_{0}^{\frac{1}{\mathcal{G}\left(\frac{1}{2C\Lambda}\right)}}\|B^* w^\prime(t)\|^2_{U} \,
dt \leq 4 \int_{0}^{\frac{1}{\mathcal{G}\left(\frac{1}{2C\Lambda}\right)}}\|B^* \phi^\prime(t)\|^2_{U} \, dt.
\end{equation}
Finally, by virtue of (\ref{444}), we conclude the following
estimate
\begin{equation}
 ||(w_{0},w_{1})||_{V\times X}^2
 \leq 16 \int_{0}^{\frac{1}{\mathcal{G}\left(\frac{1}{2C\Lambda}\right)}}\|B^* \phi^\prime(t)\|^2_{U} dt.
\end{equation}
\end{proof}

\begin{proof}[Proof of (2)]
Using similar arguments as previously, the inequality (\ref{obs2})
becomes
$$
E(w(0)) \leq  2C\int_{0}^{\frac{1}{\mathcal{G}\left(\frac{1}{2C\Lambda}\right)}}\big(\|B^*w^\prime(t)\|^2_U
+\|B^*v^\prime(t)\|^2_U\big) dt.
$$
Multiply the first equation of (\ref{second}) by $v^\prime$ and
integer by parts. It follows that
\begin{equation}
E(v(t)) \leq
\int_{0}^{t}\left(\frac{\|B^*w^\prime (s)\|^2_U}{\varepsilon}+
\varepsilon\|B^*v^\prime (s)\|^2_U\right) \, ds, \forall \, \varepsilon>0.
\end{equation}
Let $T>0$ be fixed, for all $0\leq t\leq T$, we have
\begin{eqnarray}
\sup_{0\leq t\leq T}E(v(t)) &\leq &
\int_{0}^{T}\left(\frac{\|B^*w^\prime (t)\|^2_U}{\varepsilon}+ \varepsilon \|B^*v^\prime (t)\|^2_U\right) \, dt\nonumber\\ &\leq &
\int_{0}^{T}\frac{\|B^*w^\prime (t)\|^2_U}{\varepsilon} \, dt + \varepsilon C \int_{0}^{T}E(v(t)) \, dt\nonumber\\
&\leq &
\int_{0}^{T}\frac{\|B^*w^\prime (t)\|^2_U}{\varepsilon} \, dt+
\varepsilon CT\sup_{0\leq t\leq T}E(v(t)).\nonumber
\end{eqnarray}
Choose $\varepsilon=\frac{1}{2CT}$, so we have
\begin{equation}
\sup_{0\leq t\leq T}E(v(t))\leq
4CT\int_{0}^{T}\|B^*w'\|^2_U dt.
\end{equation}
On the other hand, as
$$\| v^\prime (t)\|_X^2 \leq E(v(t)),$$
we integer on $[0,T]$, to get
\begin{eqnarray}
\int_{0}^{T}\| v^\prime (t)\|_X^2 dt
&\leq &\int_{0}^{T}E(v(t))\ dt\nonumber\\
&\leq & T\sup_{0\leq t\leq T}E(v(t))\nonumber\\
&\leq &
4CT^2\int_{0}^{T}\|B^*w'\|^2_U  dt,\nonumber
\end{eqnarray}
so
\begin{equation}\label{4Talpha}
\int_{0}^{T}\|B^*v'\|^2_U  dt\leq
4C^2T^2\int_{0}^{T}\|B^*w'\|^2_U  dt.
\end{equation}
Hence, for $T=\frac{1}{\mathcal{G}\left(\frac{1}{2C\Lambda}\right)}$ we conclude that
\begin{equation}  \label{E(w,0)}
E(w(0))\leq
2C\left(1+ 4C^2 \left(\frac{1}{\mathcal{G}\left(\frac{1}{2C\Lambda}\right)}\right)^{2}\right)\int_{0}^{\frac{1}{\mathcal{G}\left(\frac{1}{2C\Lambda}\right)}}\|B^*w^\prime\|^2_U  \, dt.
\end{equation}

One can easily verify that
$$ \Lambda \leq  \frac{E(w(0))+E(w^\prime(0))}{E(w(0))}:=\tilde{\Lambda},$$ and consequently
\begin{equation}
E(w(0))\leq
2C\left(1+4C^2\left(\frac{1}{\mathcal{G}\left(\frac{1}{2C\tilde{\Lambda}}\right)}\right)^{2}\right)
\int_{0}^{\frac{1}{\mathcal{G}\left(\frac{1}{2C\tilde{\Lambda}}\right)}}\|B^*w^\prime (t)\|^2_U \, dt.
\end{equation}
Since $\tilde{\Lambda}$ is minimized, we get
\begin{equation}
E(w(0))\leq
C \, \left(\frac{1}{\mathcal{G}\left(\frac{1}{2C\tilde{\Lambda}}\right)}\right)^{2}
\int_{0}^{\frac{1}{\mathcal{G}\left(\frac{1}{2C\tilde{\Lambda}}\right)}}\|B^*w^\prime (t)\|^2_U  \, dt.
\end{equation}
By translating the time variable and using the formula
$$ E(w(t_1))-E(w(t_2))+\int_{t_2}^{t_1}\|B^*w^\prime (t)\|^2_U  \, dt = 0,$$
we obtain:
$$
\frac{E(w(0))}{E(w(0))+ E(w^\prime(0))}
\leq
C\, \left(\frac{1}{\mathcal{G}\left(\frac{1}{2C\tilde{\Lambda}}\right)}\right)^{2}
\int_{s}^{\frac{1}{\mathcal{G}\left(\frac{1}{2C\tilde{\Lambda}}\right)}+s}\frac{\|B^*w^\prime (t)\|^2_U }{E(w(0))
+ E(w^\prime(0))} dt
$$
$$
\leq  C \, \left(\frac{1}{\mathcal{G}\left(\frac{1}{2C\tilde{\Lambda}}\right)}\right)^{2}
\left( \frac{E(w(s))}{E(w(0))+E( w^\prime(0))} -
\frac{E\left(w \left(\frac{1}{\mathcal{G}\left(\frac{1}{2C\tilde{\Lambda}}\right)}+s \right)\right)}{E(w(0))
+E(w^\prime(0))} \right).
$$
Put
$\mathcal{H}(s)=\frac{E(w(s))}{2C\left(E(w(0))+E(w^\prime(0))\right)}$.\\
Make use of the previous inequality and the decay of $\mathcal{H}$,
it follows that

$$\mathcal{H}(s)\leq  C \left(\frac{1}{\mathcal{G}(\mathcal{H}(s))}\right)^{2}
\left( \mathcal{H}(s) - \mathcal{H}\left(\frac{1}{\mathcal{G}(\mathcal{H}(s))}+s \right)\right), \forall s>0.$$

Thanks to (\ref{h}) and Lemma \ref{lemme2}, there exists $C$ such
that
$$\frac{E(w(t))}{E(w(0))+ E(\partial_t w(0))}\leq C
\mathcal{F}^{-1}\left(\frac{1}{\sqrt{t}}\right).$$ We conclude the
desired result.
\end{proof}

\section{Some applications}\label{appls}
We give some applications of Theorem \ref{prop1}.
\subsection{The linear case}

\subsubsection{Example 1}
Let $\mathcal{G}$ be given by $\mathcal{G}(x)=x^p$ on $(0,r_0]$,
$r_0 >0$ and $p\in \mathbb{R}\setminus[-\frac{1}{2},0]$. Then The
following two statements are equivalent.
\begin{itemize}
\item[i)] There exists $C>0$ such that for all  non-identically zero initial data $(w_0,w_1)\in V\times X$, the solution $\phi$ of
\textsc{(\ref{eq3})} satisfies:
\begin{equation*}
||(w_{0},w_{1})||_{V \times
X}^2\leq
16 \int_{0}^{C\Lambda^p}\|B^* \phi^\prime(t)\|^2_U \, dt.
\end{equation*}
\item[ii)] There exists $C>0$ such that for all non-identically zero initial data
$(w_{0},w_{1})\in \mathcal{D}(A) \times V$ and for all $t>0$, the
solution $w$ of \rfb{eq1} satisfies:
\begin{equation*}
E(w(t))\leq \frac{C}{t^p}||(w_{0},w_{1})||_{\mathcal{D}(A) \times V}^2.
\end{equation*}
\end{itemize}
\begin{remark}
In \cite{phung1}, the author construct a geometry with a trapped ray  for the linear dissipative wave equation (the geometric control condition is then not fulfilled) and establish a polynomial decay rate when $(w_0,w_1) \in \left[H^2(\Omega)\cap H_0^1(\Omega) \right] \times H^1_0(\Omega)$, the estimate (\ref{obs1}) is satisfied for $\mathcal{G}(x)=x^\delta,$ $\delta>0$.
\end{remark}
\subsubsection{Example 2}
Let $\mathcal{G}$ be given by
$\mathcal{G}(x)=\frac{exp(-\frac{1}{x^p})}{\sqrt{x}}$  on $(0,r_0]$,
$p\in \mathbb{R}_+$. The following statements hold.
\begin{itemize}
\item[i)] The existence of a constant $C>0$ such that the solution $\phi$ of
\rfb{eq3} satisfies:
\begin{equation*}
||(w_{0},w_{1})||_{H_{0}^{1}(\Omega)\times L^{2}(\Omega)}^2\leq
16\int_{0}^{\frac{1}{\mathcal{G}\left(\frac{1}{2C\Lambda}\right)}}
\|B^* \phi^\prime(t)\|^2_U\, dt,
\end{equation*}
implies the existence of a constant $C_1>0$ such that for all
non-identically zero initial data $(w_{0},w_{1})\in \mathcal{D}(A) \times V$ and for all $t>0$,
the solution $w$ of \rfb{eq1}
satisfies:
\begin{equation*}
E(w(t))\leq \frac{C_1}{(\ln  t)^{\frac{1}{p}}}||(w_{0},w_{1})||_{\mathcal{D}(A) \times V}^2.
\end{equation*}
\item[ii)] The existence of a constant $C_1>0$ such that for all non-identically $(w_{0},w_{1})
\in \mathcal{D}(A) \times V $ and for all $t>0$, the solution $w$ of \rfb{eq1}
satisfies:
\begin{equation*}
E(w(t))\leq \frac{C_1}{(\ln  t)^{\frac{1}{p}}}||(w_{0},w_{1})||_{\mathcal{D}(A) \times V}^2.
\end{equation*}
implies the existence of a constant $C>0$ such that the solution $\phi$ of
\rfb{eq3} satisfies:
\begin{equation*}
||(w_{0},w_{1})||_{V \times X}^2 \leq
16 \, \int_{0}^{\frac{1}{\mathcal{F}\left(\frac{1}{2C\Lambda}\right)}}\|B^* \phi^\prime(t)|^2 \, dt.
\end{equation*}

\end{itemize}
\subsection{The nonlinear case}
Let $\Omega$ be a bounded connected open set of $\mathbb{R}^n$,
$n>1$ with a $\mathcal{C}^2$ boundary $\partial \Omega$. Let also $M
= \big( \alpha^{ij} \big)_{ 1\leq i,j \leq n } \in
\mathcal{C}^{\infty} \big( \bar{\Omega}; \mathbb{R}^{n\times
n}\big)$ be a symmetric and uniformly positive definite matrix.
Denote by $\nabla = \Big( \ds \Sigma_{j=1}^n
\beta^{1j}\partial_{x_j}, ...,\Sigma_{j=1}^n
\beta^{nj}\partial_{x_j}\Big)$ and $\Delta = \ds \Sigma_{i,j=1}^n
\partial_{x_i}\big(\alpha^{ij}\partial_{x_j}\big)$. We deal with
the following second order differential equation:

\be \label{nonli}
\left\{
\begin{array}{ll}
\partial_t^2 u -\Delta u + a(x)g(\partial_t u) = 0, \mbox{ in } \Omega \times (0,+\infty),\\
u = 0, \mbox{ on } \partial \Omega \times (0,+\infty),\\
(u,\partial_t u)(.,0)=(u_0,u_1), \mbox{ in } \Omega,
\end{array}
\right.\ee

where $a = a(x) \in L^{\infty}(\Omega)$ is a bounded function with
$a(x)\geq 0$ for all $x \in \Omega$ and $g : \mathbb{R}\rightarrow
\mathbb{R}$ is a continuous strictly increasing function with
$g(0)=0,$ $sg(s)\geq 0$. We assume the additional conditions:
\begin{itemize}
\item[(i)] $\exists r\in [1,\infty)$, $\exists c_1, c_2 > 0$, $|s|\leq 1 \Rightarrow$
$c_1 |s|^r \leq |g(s)| \leq c_2 |s|^{1/r}$.
\item[(ii)] $\exists k \in [0,1]$, $\exists p \in [1,\infty)$, $\exists c_3,$ $c_4 > 0$,
$|s| > 1$ $\Rightarrow$ $c_3 |s|^k\leq |g(s)| \leq c_4 |s|^p.$
\item[(iii)] $(n - 2)(1-k) \leq 4r$ and $(n - 2) (p - 1) \leq  1.$
\end{itemize}
For $(u_0,u_1)\in H^1_0(\Omega)\times L^2(\Omega)$, there exists a
unique solution $u \in
\mathcal{C}\Big([0,+\infty),H^1_0(\Omega)\Big) \cap
\mathcal{C}^1\Big([0,+\infty),L^2(\Omega)\Big)$. For more regular
initial data $(u_0, u_1) \in \left[H^2(\Omega)\cap
H^1_0(\Omega)\right] \times H^1_0(\Omega),$ the solution $u$ has the
following regularity $u \in L^\infty \Big(0,+\infty;H^2(\Omega)\cap
H^1_0(\Omega)\Big)
\cap W^{1,\infty}\Big( 0,+\infty;H^1_0(\Omega) \Big)\cap W^{2,\infty}\Big( 0,+\infty;L^2(\Omega) \Big).$\\
The energy of a solution is defined at instant $t \geq 0$ by
$$ E (u(t)) = \frac{1}{2} \int_{\Omega} \Big( |\partial_t u(x,t)|^2 + |\nabla u (x, t)|^2\Big)dx.$$
$E(u(t))$ is a non-increasing function of time and satisfies, for
all $t_2 > t_1 \geq 0$ the identity
$$ E (u(t_2)) - E (u(t_1)) =-\int_{t_1}^{t_2}\int_{\Omega} a(x)g (\partial_t u (x, t))
\partial_t u (x, t) dxdt \leq 0.$$
Denote by
$$X(u_0, u_1) = E (u(0))+E_1 (u(0))+[E_1(u(0))]^{(2p-1)}+[E_1(u(0))]^{(1+\frac{r-k}{r+1})},$$
where
$$
E_1(u(0)) = \|(\Delta u_0 - ag(u_1),u_1)\|^2_{L^2(\Omega) \times H^1_0(\Omega)}.
$$
We introduce $u_l$ the solution of the linear locally damped
problem: \be \left\{
\begin{array}{ll}
\partial^2_t u_l -\Delta u_l + a(x)\partial_t u_l = 0, \mbox{ in } \Omega \times (0,+\infty),\\
u_l = 0, \mbox{ on } \partial \Omega \times (0,+\infty),\\
(u_l,\partial_t u_l)(.,0)=(u_0,u_1) \in \left[H^2(\Omega)\times H^1_0(\Omega) \right] \cap H^1_0(\Omega),
\end{array}
\right.\ee
and make the following assumption:
\begin{itemize}
\item[\bf{(A)}] Assume that $x\mapsto x \, \mathcal{F}^{-1}(\frac{1}{x})$
is an increasing function and there exists $C>0$ such that for all
non-identically zero initial data $(u_{0},u_{1})\in
\left[H^2(\Omega)\times H^1_0(\Omega) \right] \cap H^1_0(\Omega) $,
the solution $\phi$ of \be \left\{
\begin{array}{ll}
\partial^2_t\phi(t)-\Delta \phi(t) = 0, \\
\phi(0) = u_0, \, \partial_t \phi (0) = u_1,
\end{array}
\right.\ee
satisfies:
\begin{equation}\label{obs2b}
||(u_{0},u_{1})||_{H^1_0(\Omega)\times L^2(\Omega)}^2 \leq
C \, \int_{0}^{\frac{1}{\mathcal{G}\left(\frac{1}{2C\Lambda_r}\right)}}
\int_{\Omega}a(x)|\partial_t\phi(x,t)|^2dxdt,
\end{equation}
where $$\Lambda_r=\frac{(r-1)+X(u_0,u_1)}{E(u(0))}.$$
\end{itemize}
The following result is deduced from Theorem \ref{prop1} and
\cite[Proposition 3]{phung2}.
\begin{proposition} Let {\bf{(A)}} holds. There exists $c> 0$ such that for
any $(u_0, u_1)\in \left[H^2(\Omega)\times H^1_0(\Omega) \right] \cap H^1_0(\Omega)$,
the solution $u$ of (\ref{nonli}) satisfies
\begin{eqnarray}
 E (u(s)) &&\leq ch((r - 1) + X (u_0, u_1))\\
&&+ c \int^{s+\frac{1}{G(h)}}_s \int_{\Omega} a(x)g(\partial_t u(x, t)) \partial_t u (x, t) dxdt ,
\end{eqnarray}
for any $h > 0$ and any $s \geq 0$ where
$$ G(h): = C {h}^{(2r+1)}\mathcal{F}(h)^{4(r+1)}.$$
\end{proposition}
We have the following stabilization result for the nonlinear damped wave equation.
\begin{theorem}
Let {\bf{(A)}} holds and suppose that there exists $c_0$ such that
the function $G$ satisfies $G^{-1}(x)\geq
\frac{c}{c+1}G^{-1}\big(x(c_0+1)\big)$ for all $x\geq 0$. Then the
energy of the solution of (\ref{nonli}) satisfies the estimate:
\begin{equation}
E(u(t)) \le  C G^{-1}\left(\frac{c'}{t}\right)\big((r-1)+X(u_0,u_1)\big), \mbox{
for } t \mbox{ sufficiently large,}
\end{equation}
and all non-identically zero initial data $(u^0,u^1) \in \left[H^2(\Omega)
\cap H^1_0(\Omega)\right] \times H^1_0 (\Omega)$, the constant $C$ depend on the initial data $(u^0,u^1) $.
\end{theorem}
\begin{proof}
Choosing
$$ h =  \frac{1}{2C \Lambda_r},$$
this implies the existence of a constant $c > 0$ such that
\begin{equation}\label{Y}
E (u(s)) \leq c \int^{s+\frac{1}{G( \frac{1}{2C \Lambda_r})}}_s \int_{\Omega}
a(x)g(\partial_t u(x, t)) \partial_t u (x, t) dxdt.
\end{equation}
Denoting by $\mathcal{H}(s) = \frac{E(u,s)}{(r-1)+X(u_0,u_1)}$, we deduce from (\ref{Y}) that
$$\mathcal{H}\left(s+ \frac{1}{G(\mathcal{H}(s))}\right) \le \mathcal{H}(s) \leq c\,
\left( \mathcal{H}(s) - \mathcal{H}\left(s+ \frac{1}{G(\mathcal{H}(s))}\right) \right),$$ which gives
$$\mathcal{H}\left(s+ \frac{1}{G(\mathcal{H}(s))}\right) \leq \frac{c}{c+1}\mathcal{H}(s).$$
\begin{itemize}
\item[$\bullet$] If $c_0s \leq c  \frac{1}{G(\mathcal{H}(s))}$, then $\mathcal{H}(s)
\leq G^{-1}\Big( \frac{c}{c_0s}\Big)$ and
\be\label{e1}\mathcal{H}((1+c_0)s)\leq \mathcal{H}(s)\leq G^{-1}\Big( \frac{c}{c_0s}\Big).\ee
\item[$\bullet$] If $c_0s > c  \frac{1}{G(\mathcal{H}(s))}$, then
\be\label{e2}\mathcal{H}((1 + c_0) s)\leq  \mathcal{H}\left(s+ \frac{1}{G(\mathcal{H}(s))}
\right)\leq \frac{c}{c+1}\mathcal{H}(s).\ee
\end{itemize}
By induction, we deduce from (\ref{e1}) and (\ref{e2}) that
$\forall  s > 0$ and $\forall n \in N^{*}$,
$$
\mathcal{H}((1 + c_0) s) \leq \max \left[ G^{-1}\Big( \frac{c}{c_0s}\Big),
\frac{c}{c+1}G^{-1}\Big( \frac{c(c_0+1)}{c_0s}\Big),...,
\right.
$$
$$
\left.
\left(\frac{c}{c+1}\right)^nG^{-1}\Big( \frac{c(c_0+1)^n}{c_0s}\Big),\left(\frac{c}{c+1} \right)^{n+1}\mathcal{H}\left(\frac{s}{(c_0+1)^{n+1}}\right)\right].
$$
Now, remark that with the above hypothesis on the function $G$,
$$\frac{c}{c+1}G^{-1}\Big( \frac{c(c_0+1)}{c_0s}\Big) \leq  G^{-1}\Big( \frac{c}{c_0s}\Big).$$
Consequently,
\begin{eqnarray}
\mathcal{H}((1 + c_0) s) &&\leq \max \left[ G^{-1}\Big( \frac{c}{c_0s}\Big),\left(\frac{c}{c+1}\right)^{n+1}\mathcal{H}\left(\frac{s}{(c_0+1)^{n+1}}\right)\right],  \nonumber \\
&&\leq \max \left[ G^{-1}\Big(
\frac{c}{c_0s}\Big),\left(\frac{c}{c+1}\right)^{n+1}\right], \hspace{16mm}\forall
n\geq 1,
\end{eqnarray}
and we conclude that
$$\mathcal{H}(s) \le  G^{-1}\left( \frac{c(1 + c_0)}{c_0s}\right), \forall s>0. $$
\end{proof}

\begin{rmks}\mbox{}
\begin{enumerate}
\item For $\mathcal{G}(x)= x^p$, we have $\mathcal{F}(x)=x^{2p+1}$ and
$G(x)=x^{(4p+3)(2r+1)-1}.$ \\ The energy of the solution of
(\ref{nonli}) satisfies the estimate:
\begin{equation}
E(u(t)) \le  \frac{c}{t^{\frac{1}{(4p+3)(2r+1)-1}}}\big((r-1)+X(u_0,u_1)\big),  \mbox{
for } t \mbox{ sufficiently large.}
\end{equation}
\item For the wave equation with arbitrary localized nonlinear damping, we obtain in \cite{nodea} a weak observability which implies in particular the estimate (\ref{obs2b}) and the logarithmic decay of the energy.\\
At the same time, this gives a geometry where the observability estimate (\ref{obs2b}) is satisfied and simplify the proof of the decay result in \cite{bella}.
\end{enumerate}
\end{rmks}


\begin{thebibliography}{99}
\bibitem{am-al} F. Alabau-Boussouira and K. Ammari, Sharp energy estimates for nonlinearly locally damped PDE's via observability for the associated undamped system, {\em J. Funct. Anal,} {\bf 260} (2011), 2424--2450.
\bibitem{ammari} K. Ammari and M. Tucsnak, Stabilization of second order evolution equations by a class of unbounded feedbacks, {\em ESAIM COCV.}, {\bf 6} (2001), 361--386.
\bibitem{kaisserge} K. Ammari and S. Nicaise, Stabilization of elastic systems by collocated feedback, {\em Lectures Notes in Mathematics}, 2124, Springer, Cham, 2015.
\bibitem{nodea} K. Ammari, A. Bchatnia and K. El Mufti, Stabilization of the nonlinear damped wave equation via linear weak observability, to appear in {\em Nonlinear Differential Equations and Applications.}

\bibitem{bella} M. Bellassoued, Decay of solutions of the wave equation with arbitrary localized nonlinear damping, {\em J. Differential Equations,} {\bf{ 211}} (2005), 303--332.
\bibitem{phung1} K.-D. Phung, {\em Polynomial decay rate for the dissipative wave equation}. {\em J. Differential Equations,} {\bf 1} (2007), 92--124.
\bibitem{phung2} K.-D. Phung, {\em Decay of solutions of the wave equation with localized nonlinear damping and trapped rays,} {\em Mathematical Control and Related Fields,} {\bf 1} (2011), 251--265.

\end{thebibliography}
\end{document}